\documentclass{amsart}
\usepackage{amssymb}
\usepackage{graphicx}

\title[Minimal Mahler Measure in Real Quadratic Fields]{Minimal Mahler Measure in Real Quadratic Fields}
\author[Cochrane, Dissanayake, Donohoue, Ishak, Pigno, Pinner \&  Spencer]{Todd Cochrane, R. M. S. Dissanayake, Nicholas Donohoue,  M. I. M. Ishak, Vincent Pigno, Chris  Pinner, \and Craig Spencer}

\address{ Department of Mathematics\\
         Kansas State University\\
         Manhattan, KS 66506}
\email{cochrane@math.ksu.edu, donohoue@ksu.edu, pinner@math.ksu.edu \& cvs@math.ksu.edu}

\address{Department of Engineering Mathematics\\
Faculty of Engineering\\ University of Peradeniya\\  Sri Lanka}
\email{sumanad@pdn.ac.lk \& mimishak@pdn.ac.lk}

\address{ Department of Mathematics \& Statistics\\
University of California, Sacramento\\
Sacramento, CA 95819}
\email{vincent.pigno@csus.edu}

\thanks{N. Donohoue  and V. Pigno were supported in part by an I-Center
Undergraduate Scholarship and an I-Center Graduate Scholarship,
respectively, and C. Spencer was supported in part by NSA Young
Investigator Grant \#H98230-14-1-0164.}

\date{\today}

\newcommand{\be}{\begin{equation}}
\newcommand{\ee}{\end{equation}}
\newcommand{\ba}{\begin{align}}
\newcommand{\ea}{\end{align}}
\newtheorem{theorem}{Theorem}[section]
\newtheorem{corollary}{Corollary}[section]
\newtheorem{lemma}{Lemma}[section]
\newtheorem{conjecture}{Conjecture}[section]

\begin{document}
\begin{abstract}
We consider upper and lower bounds on the minimal height 
of an irrational number lying in a particular real quadratic field.
\end{abstract}

\maketitle

\section {Introduction }
 
For a polynomial $F(x)=a_n \prod_{i=1}^n (x-\alpha_i)$ in $\mathbb C [x]$ one defines its
Mahler measure $M(F)$  as 
$$ M(F) =|a_n| \prod_{i=1}^n \max\{ 1,|\alpha_i|\}. $$
For an algebraic number $\alpha$ we use  $M(\alpha)$ to denote the Mahler measure of an irreducible integer polynomial with root $\alpha$. Thus the logarithmic Weil height of $\alpha$
can be written 
$$ h(\alpha) = \frac{\log M(\alpha)}{[\mathbb Q (\alpha) : \mathbb Q]}. $$
Of course $M(\alpha)=1$ iff $\alpha$ is a root of unity and the well known problem of Lehmer
\cite{lehmer}  is to determine whether there is a constant $C>1$  such that $M(\alpha)>C$ otherwise.
Schinzel \cite{Schinzel} showed that for  $\alpha$ in a Kroneckerian field (a totally real field or a quadratic extension
of such a field) the value of $M(\alpha)$ must in fact grow with its degree, with the absolute minimum $M(\alpha)>1$ achieved for the golden ratio
$$ M\left( \frac{1+\sqrt{5}}{2}\right) =\frac{1+\sqrt{5}}{2} . $$
Amoroso \& Dvornicich \cite{amoroso}  further extended this to cyclotomic fields. These of course include the quadratic fields $\mathbb Q(\sqrt{d})$, where $d$ is a square-free positive integer. 
Since the golden ratio is not in all these fields we are interested in how 
$$ L(d) := \min \left\{ M(\alpha)\; :\; \alpha \in \mathbb Q (\sqrt{d})\setminus \mathbb Q\right\}$$
varies with $d$.
We recall the discriminant of the field $\mathbb Q (\sqrt{d})$
$$ D:= \begin{cases}  d, & \text{ if $d\equiv 1$ mod 4,}\\
 4d, & \text{if $d\equiv 2$ or 3 mod 4.} \end{cases} $$
Since $M(\alpha)=M(-\alpha)=M(\overline{\alpha})$  we assume that our $\alpha \in  \mathbb Q (\sqrt{d})\setminus \mathbb Q$ takes the form
\be \label{defal} \alpha = \frac{a+b\sqrt{d}}{c}, \;\; a,b,c\in \mathbb Z,\; a\geq 0,\, b>0,\,c>0,\;  \gcd(a,b,c)=1,\ee
with conjugate
 $$\overline{\alpha} = \frac{a-b\sqrt{d}}{c}, $$ 
and
\be  \label{MM} M(\alpha) =k \max\{ 1,|\alpha|\}\max\{ 1,|\overline{\alpha}|\}  \ee
where $k$ is the smallest positive integer such that 
\be \label{defk}
 k(x-\alpha)(x-\overline{\alpha})=k\left( x^2 -\frac{2a}{c} x + \frac{a^2-b^2d}{c^2}\right) \in \mathbb Z[x]. 
\ee
We show that the minimal measure must grow with $d$:

\begin{theorem}\label{main}  For square-free $d$ in $\mathbb N$
$$ \frac{1}{2}\sqrt{D}  <  L(d)   < \sqrt{D}. $$
\end{theorem}

\section {Proof of  Theorem \ref{main}}

The upper bound follows at once from the following constructive examples:

\begin{lemma} \label{ex}
Suppose that $d\geq 2$ is a square-free positive integer and let $m$ be the integer in $(\sqrt{d}-2,\sqrt{d})$ with the same parity as $d$. Then
$$ M\left(    \frac{m+\sqrt{d}}{2} \right) = \begin{cases} 2, & \text{if $d=2$,}\\
\frac{1}{2}(\sqrt{d}+m), & \text{ if $d\equiv 1$ mod 4,} \\
\sqrt{d}+m, & \text{ otherwise.} \end{cases}$$
\end{lemma}

\begin{proof} 

Observe that $\alpha=(m+\sqrt{d})/2$ has $-1<\overline{\alpha} <0$, with  $\alpha >1$ 
for $d\geq 3$  (and $0< \alpha <1$ for $d=2$). 
The minimal $k$ to make
$k(x-\alpha)(x-\overline{\alpha})=k\left( x^2 -mx + \frac{1}{4}(m^2-d)\right)$
an integer polynomial is plainly $k=1$ if $d\equiv 1$ mod 4 and $k=2$
for $d=2$ or 3 mod 4, and the claim  is clear from \eqref{MM}.

\end{proof}

For the lower bound we first observe that $c$ or $c/2$ must divide the lead coefficient $k$.

\begin{lemma} Suppose that $d\geq 2$ is squarefree and $\alpha$ is of the form \eqref{defal}.
Suppose that  $k(x-\alpha)(x-\overline{\alpha})$ is in $\mathbb Z[x]$. 

 If $c$ is even and $d\equiv 1$ mod 4  then  $c/2 \mid k$   with $k=c/2$
iff $a,b$ are odd with $2c\mid a^2-db^2$. If $c$ is odd or $d\equiv 2$ or 3 mod 4 then $c\mid k$
with $k=c$ iff $c\mid a^2-db^2$. 
\end{lemma}

\begin{proof} Suppose that $p^t \mid\mid c$ with $t\geq 1$.

 For
$k(a^2-db^2)/c^2$ to be  in $\mathbb Z$ we must have $p^{t+1}\mid k$ unless $p^t\mid a^2-db^2$
and  $p^t\mid k$ unless  $p^{t+1}\mid a^2-db^2$. 

Hence we can assume that $p^{t+1} \mid a^2-db^2$.  Notice that in this case $p\nmid a$; since $p\mid a$
and $p^2\mid a^2-db^2$
would imply $p^2\mid db^2$, but $\gcd(a,b,c)=1$ means $p\nmid b$ and $d$ is squarefree.
In particular this case can not happen when $p=2$ and $d\equiv 2$ or 3 mod 4 (since $a^2-db^2\not \equiv 0$ mod 4),  and $a,b$ must be odd if $d\equiv 1$ mod 4.  Hence $2ka /c$ in $\mathbb Z$ forces $p^t\mid k$  when $p$ is odd and $2^{t-1}\mid k$ when $p=2$ and $d\equiv 1$ mod 4.

\end{proof}

The following lemma completes the proof of the lower bound:

\begin{lemma}
 Suppose that $d\geq 2$ is squarefree and $\alpha$ is of the form \eqref{defal}. Then
$$ M(\alpha) > \frac{1}{2} \sqrt{D}. $$
Moreover 
$$M(\alpha)>\sqrt{D} $$ 
unless  $b=1$ and $a<\sqrt{d}$, with $c\mid a^2-d$ if $d\equiv 2$ or 3 mod 4
and with $c$ even and $2c\mid a^2-d$  if $d\equiv 1$ mod 4.
\end{lemma}

\begin{proof}  Observing that 
$$ 2\frac{b\sqrt{d}}{c}  = \alpha -\overline{\alpha} \leq \alpha + |\overline{\alpha}| <  2\alpha, $$
we have 
$$ M(\alpha) \geq k \alpha  > k\frac{b\sqrt{d}}{c},$$
and the bound follows from $k\geq c$  if $c$ is odd or $d\equiv 2$ or 3 mod 4 
(with $k\geq 2c$ if $c\nmid a^2-db^2$),  and  $k\geq c/2$ if $c$ is even and $d\equiv 1$ mod 4
(with $k\geq c$ if $2c\nmid a^2-db^2$). If $a\geq \sqrt{d}$ then $M(\alpha) \geq k\alpha >\sqrt{D}$.

\end{proof}

\section{Computations}

Hence $\frac{1}{2}\sqrt{D} < L(d) <\sqrt{D}$, and an $\alpha$ of the form \eqref{defal}
with  $\frac{1}{2}\sqrt{D} < M(\alpha) <\sqrt{D}$  must be of the form 
$$ \alpha = \frac{a+\sqrt{d}}{c}, \; a<\sqrt{d},$$
with  $c\mid a^2-d$  if $d\equiv 2$ or 3 mod 4, and $c$  even with $2c\mid a^2-d$  if $d\equiv 1$ mod 4. Since $|\overline{\alpha}|\leq \alpha$  we have  $M(\alpha)=k \max\{ 1, \alpha, \alpha |\overline{\alpha}|\}$,  and in these cases we have
\be \label{smallmeas} M(\alpha)=\varepsilon \max\left\{ c, a+\sqrt{d},\frac{d-a^2}{c}\right\},\ee
where
\be \label{defve} \;\; \varepsilon:=\begin{cases} 1, & \text{ if $d\equiv 2$ or 3 mod 4,} \\ \frac{1}{2}, & \text{ if $d\equiv 1$ mod 4.}\end{cases} \ee

\begin{figure}[!h]
\centering
\includegraphics[scale=.56]{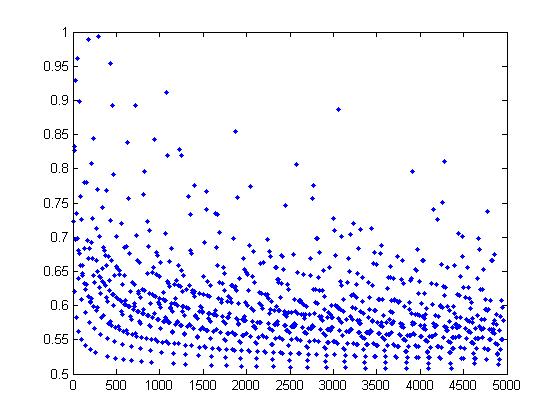}
\caption{$L(d)/\sqrt{D}$ for $d \equiv 1$ mod $4$ less than five thousand.  }
\label{fig:onemod4s}
\end{figure}
\begin{figure}[!h]
\centering
\includegraphics[scale=.56]{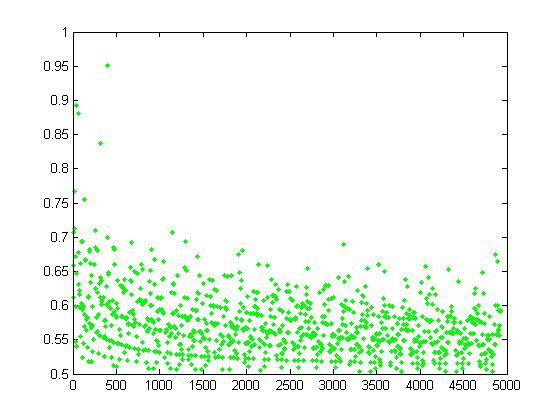}
\caption{$L(d)/\sqrt{D}$ for $d \equiv 2$ mod $4$ less than five thousand.  }
\label{fig:twomod4s}
\end{figure}
\vspace{2cm}
\begin{figure}[!h]
\centering
\includegraphics[scale=.56]{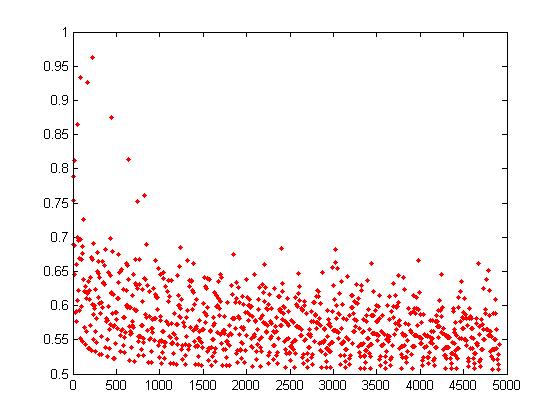}
\caption{$L(d)/\sqrt{D}$ for $d \equiv 3$ mod $4$ less than five thousand.  }
\label{fig:threemod4s}
\end{figure}
\begin{figure}[!h]
\centering
\includegraphics[scale=.56]{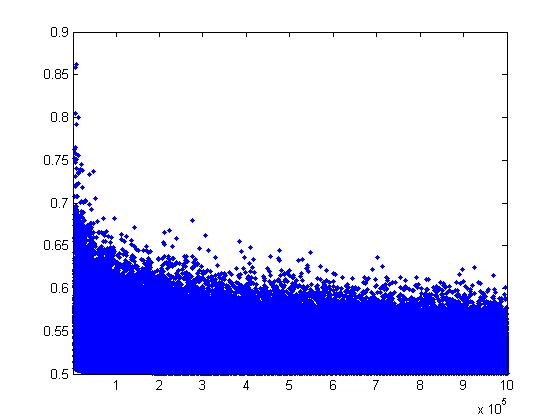}
\caption{$L(d)/\sqrt{D}$ for $d \equiv 1$ mod $4$ between five thousand and one million.  }
\label{fig:onemod4m}
\end{figure}
\begin{figure}[!h]
\centering
\includegraphics[scale=.56]{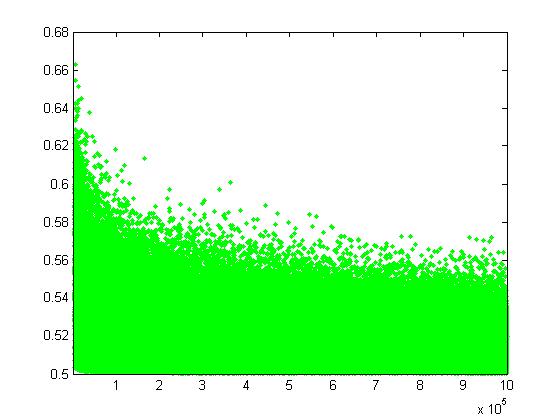}
\caption{$L(d)/\sqrt{D}$ for $d \equiv 2$ mod $4$ between five thousand and one million.  }
\label{fig:twomod4m}
\end{figure}
\begin{figure}[!h]
\centering
\includegraphics[scale=.56]{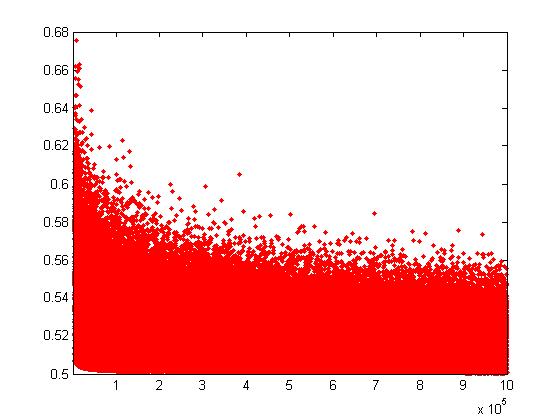}
\caption{$L(d)/\sqrt{D}$ for $d \equiv 3$ mod $4$ between five thousand and one million.  }
\label{fig:threemod4m}
\end{figure}
\clearpage  
\begin{figure}[!h]
\centering
\includegraphics[scale=.56]{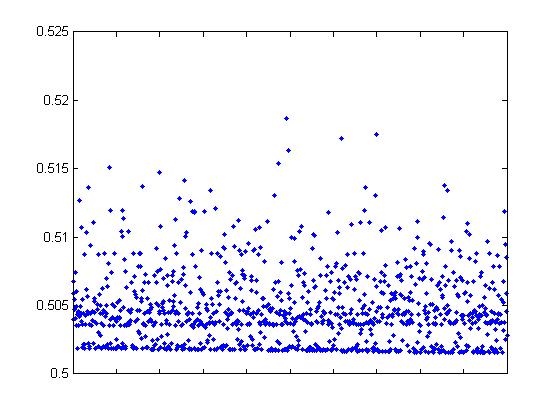}
\caption{$L(d)/\sqrt{D}$ for $d$ between one billion and one billion five thousand.  }
\label{fig:onemod4h}
\end{figure}
\begin{figure}[!h]
\centering
\includegraphics[scale=.56]{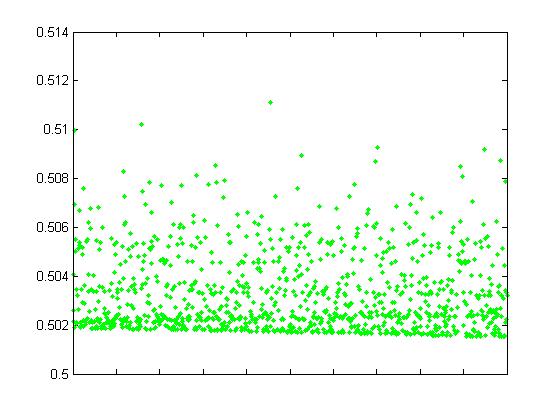}
\caption{$L(d)/\sqrt{D}$ for $d\equiv 2$ mod 4  between one billion and one billion five thousand.  }
\label{fig:twomod4h}
\end{figure}
\begin{figure}[!h]
\centering
\includegraphics[scale=.56]{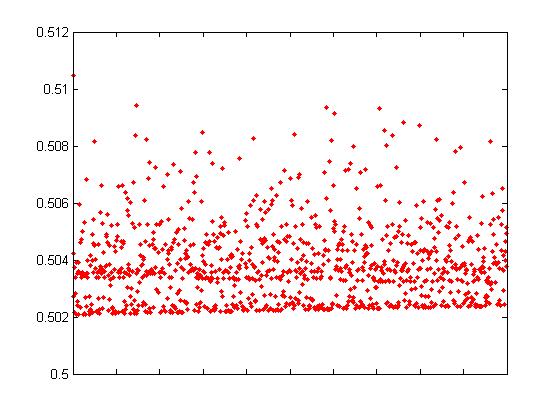}
\caption{$L(d)/\sqrt{D}$ for $d\equiv 3$ mod 4  between one billion and one billion five thousand.  }
\label{fig:threemod4h}
\end{figure}

\section{How good are our bounds?}
Theorem \ref{main} tells us  that
$$ \frac{1}{2} < \frac{L(d)}{\sqrt{D}} < 1. $$
Figures \ref{fig:onemod4h},\ref{fig:twomod4h}  \& \ref{fig:threemod4h} make it seem  reasonable to make the following conjecture:

\begin{conjecture}\label{limit}
$$ \lim_{d\rightarrow \infty}  \frac{L(d)}{\sqrt{D}}\ = \frac{1}{2}.$$
\end{conjecture}

\noindent
In view of  \eqref{smallmeas} this can be equivalently written:

\begin{conjecture}
For any square-free positive integer $d$ 
 there exists  an   $a$ and $c$ with
$$ a=o(\sqrt{d}), \;\; c=(1+o(1))\sqrt{d} , $$
and $c\mid d-a^2$ when $d\equiv 2$ or 3 mod 4,  $c$ even and $2c\mid d-a^2$ 
when $d\equiv 1$ mod 4.
\end{conjecture}

Checking computationally,   pairs $a$ and $c$ satisfying 
$$a < d^{2/5},\;\;\; d^{1/2}-d^{2/5} <c < d^{1/2}+d^{2/5}$$ 
and $c\mid d-a^2$  exist for all  $827 < d < 2,000,000,000$, and even $c$ with $2c\mid d-a^2$
for all  $d\equiv 1$ mod 4 with  $1,902,773<d<2,000,000,000$.

\vspace{2ex}

The $\frac{1}{2}$ in the lower bound is the optimal absolute constant.

\begin{theorem}
$$ \liminf_{d\rightarrow \infty}  \frac{L(d)}{\sqrt{D}}=\frac{1}{2}. $$
\end{theorem}

This follows at once from the following examples:

\vspace{1ex}
\noindent
{\bf Small Examples.} If   $d=m^2+1$ then 
$$ M\left( \frac{\sqrt{d}+1}{m}\right)=\begin{cases}  \sqrt{d}+1 , & \text{ if $m$ is odd,}\\
 \frac{1}{2} \left( \sqrt{d}+1\right), & \text{ if $m$ is even.}  \end{cases}$$

It seems likely that the upper bound can be slightly reduced. The computations suggest  that 
the largest value occurs at $d=293$.

\begin{conjecture}
$$   \sup_d  \frac{L(d)}{\sqrt{D}} = \frac{L(293)}{\sqrt{293}}=\frac{ M\left( \frac{\sqrt{293}+15}{2}\right)}{\sqrt{293}} = \frac{17}{\sqrt{293}}=0.993150\ldots . $$
\end{conjecture}

If we separate out  the residue classes mod 4:

\begin{conjecture}
\begin{align*}  \sup_{d \equiv 2 \text{ mod 4}}  \frac{L(d)}{\sqrt{D}}  & = \frac{L(398)}{2\sqrt{398}}=\frac{ M\left( \frac{\sqrt{398}+18}{2}\right)}{2\sqrt{398}} = \frac{\sqrt{398}+18}{2\sqrt{398}}=0.951129\ldots . \\
  \sup_{d \equiv 3 \text{ mod 4}}  \frac{L(d)}{\sqrt{D}}  & = \frac{L(227)}{2\sqrt{227}}=\frac{ M\left( \frac{\sqrt{227}+13}{2}\right)}{2\sqrt{227}} = \frac{29}{2\sqrt{227}}=0.962398\ldots . 
\end{align*}
\end{conjecture}

\vspace{1ex}
We found only ten values of $d$, namely $d=293,173,227,53,437,398,83,29,167,1077$,  with $L(d)/\sqrt{D}>0.9$. As can be seen in the Appendix, the large  values
on each  of the Figures  \ref{fig:onemod4s}, \ref{fig:twomod4s} \& \ref{fig:threemod4s} noticeably seem to correspond to $d$ with the property that $d$ is  a quadratic non-residues for all small primes $p\nmid d$
(specifically all  $p<\sqrt{d}$ for $d=2$ or 3 mod 4 and $p<\sqrt{d}/2$ for $d=1$ mod 4).
Most  of these  $d$  (with the exception of 437)  are of the form $d=\ell p$ with $p$ prime and $\ell$ small, and all have  $d\not\equiv 1$ mod 8.
 The following lemma shows why such $d$ have large $L(d)$ values.

\begin{lemma}\label{nonresidue}
 Suppose that $d$ is a squarefree positive integer with $d\equiv 2$  (mod 4) or 3 (mod 4)  or 5  (mod 8), and that $\left(\frac{d}{p}\right)=-1$ for
all primes $p\nmid d$ with
$$ p<\begin{cases} \sqrt{d}, & \text{ if $d\equiv 2$ or 3 (mod 4),} \\ \sqrt{d}/2,& \text{ if $d\equiv 5$ (mod 8).} \end{cases}$$
 For each odd $A\mid d$,  with $A<\sqrt{d}$,  let $m_A$ denote the integer
in $\left(\sqrt{d}/A-2,\sqrt{d}/A\right)$ with the same parity as $d$. Then, with $\varepsilon$ as in \eqref{defve},
\begin{align*}  L(d) & =\varepsilon \min_{A\mid d,A<\sqrt{d} \text{ odd}}\min\left\{ M\left( \frac{\sqrt{d}+m_AA}{2A}\right),\;  M\left( \frac{\sqrt{d}+(m_A-2)A}{2A}\right) \right\} \\
 & =\varepsilon \min_{A\mid d,A<\sqrt{d} \text{ odd}}\min\left\{ \sqrt{d}+m_AA, \; (d-(m_A-2)^2A^2)/2A  \right\} \\
 & \geq \sqrt{D}-2\varepsilon  \max_{A\mid d, A<\sqrt{d} \text{ odd}} A. 
\end{align*}

\end{lemma}

\begin{proof} Suppose that $d\equiv 2$ or 3 (mod 4). Since $\left(\frac{d}{p}\right)=-1$ for all $p<\sqrt{d}$, $p\nmid d$ we have
$d-a^2=A_1$ or $2A_1$ or $A_1p$ or $2A_1p$ with $A_1 \mid d$ odd and $p>\sqrt{d}$ 
prime. Hence we can assume that $c\mid d-a^2$ is of the form $c=A_2$ or $2A_2$ or $A_2p$ or $2A_2p$
where $A_2\mid A_1$ and $0\leq a<\sqrt{d}$,  and
$$ M\left( \frac{a+\sqrt{d}}{c}\right)=\max\left\{ c,\; a+\sqrt{d},\; (d-a^2)/c\right\}.$$
 Hence with $A=A_1$ or $A_1/A_2$ we can assume $A\mid d$ odd, $a=kA$,
and it is enough to consider
$$ M\left(\frac{\sqrt{d}+kA}{A} \right)= \max \left\{ A, \sqrt{d}+kA, (d-k^2A^2)/A\right\} $$
with $A<\sqrt{d}$ or
$$  M\left(\frac{\sqrt{d}+kA}{2A} \right)= \max \left\{2 A, \sqrt{d}+kA, (d-k^2A^2)/2A\right\} $$
with $2A<\sqrt{d}$ and $k$ and $d$ the same parity. Hence
\begin{align} \label{case1}  M\left(\frac{\sqrt{d}+kA}{A} \right)  & =\begin{cases}  \sqrt{d}+kA, & \text{ if $\sqrt{d}/A -1 < k < \sqrt{d}/A$,} \\
(d-k^2A^2)/A, & \text{ if $k< \sqrt{d}/A-1, $} \end{cases}\\
 & \geq 2\sqrt{d}-A,  \nonumber
\end{align}
and for $k$ and $d$ the same parity
\begin{align} \label{case2} M\left(\frac{\sqrt{d}+kA}{2A} \right) & =\begin{cases}  \sqrt{d}+kA, & \text{ if $\sqrt{d}/A -2 < k < \sqrt{d}/A$,} \\
(d-k^2A^2)/2A, & \text{ if $k< \sqrt{d}/A-2, $} \end{cases} \\
 & \geq 2\sqrt{d}-2A. \nonumber
\end{align}
For $k\geq m_A$ the minimum of both  is plainly 
$$  M\left(\frac{\sqrt{d}+m_A A}{2A} \right)= \sqrt{d}+m_A A.$$
In \eqref{case2} the smallest for $k\leq m_A-2$ is
$$  M\left(\frac{\sqrt{d}+(m_A-2)A}{2A} \right)= (d-(m_A-2)^2A^2)/2A,$$
and for \eqref{case1} the smallest  for $k\leq m_A-1$ is
$$  M\left(\frac{\sqrt{d}+(m_A-1)A}{A} \right)= (d-(m_A-1)^2A^2)/A.$$
Writing $m_A=\sqrt{d}/A-\delta$, $0<\delta <2$ and observing that
$$   M\left(\frac{\sqrt{d}+(m_A-1)A}{A} \right)-   M\left(\frac{\sqrt{d}+(m_A-2)A}{2A} \right)=\delta\sqrt{d} +A\left( 1-\frac{\delta^2}{2}\right)   > \delta \left( \sqrt{d} -\frac{1}{2}A\right)>0 $$
the result follows. 

Similarly for $d\equiv 5$ mod 8 we must have $d-a^2 = 2^2A_1$ or $2^2A_1p$,  and our even $c$
with $2c\mid d-a^2$ must be of the form  $c=2A_2$ or $2A_2p$. Thus we again reduce to 
$$ M\left(\frac{\sqrt{d}+kA}{2A}\right) =\frac{1}{2}  \max\left\{ 2A, \sqrt{d}+kA, (d-k^2A^2)/2A\right\}  $$
with $A\mid d$ odd, $2A<\sqrt{d}$, $k$ odd, and the minimum occurs for $k=m_A$ or $m_A-2$ as before.
\end{proof}

Plainly $d\equiv 2$ or $3$  (mod 4) or 5 (mod 8)  with no divisors in $(o(\sqrt{d},\sqrt{d})$  that are quadratic non-residues for all $p<\sqrt{d}$ would have $L(d)\geq \sqrt{D}-o(\sqrt{d})$.
In particular infinitely many would  immediately give 
$$\limsup_{d\rightarrow \infty} \frac{L(d)}{\sqrt{D}}=1 $$
in contradiction to Conjecture \ref{limit}, but this seems unlikely:

\begin{conjecture}
All but finitely many squarefree $d$  have  $\left(\frac{d}{p}\right)=1$ for some odd  prime  $p<\frac{1}{2}\sqrt{D}$.

\end{conjecture}

Assuming  GRH for the mod $4d$ character 
$$\chi (n) := \begin{cases} \left(\frac{d}{n}\right),  & \text{ if $\gcd (n,4d)=1$,} \\
                0, & \text{ otherwise,} 
             \end{cases} $$
(where  $\left(\frac{d}{n}\right)$ denotes the Jacobi symbol),
 we have the bound
\be \label{GRH} \left| \sum_{n\leq x} \chi (n)\Lambda (n)\right| \ll  x^{\frac{1}{2}}\log^2 (Dx)\ee
 (see, for example, \cite[Chapter 20]{Davenport}), and so
we should in fact have $\left(\frac{d}{p}\right)=1$  for some  prime $p\ll \log^4 D$.

Note, a squarefree $d\not\equiv 1 $  mod 8 with $\left(\frac{d}{p}\right)=-1$ for all odd primes $p\nmid d$ with $p<\frac{1}{2}\sqrt{D}$ 
must be of the form $d=(kA)^2\pm 2A$  or $((2k-1)A)^2\pm 4A$,  for some $k$ and squarefree odd  $A\mid d$ with $A<\frac{1}{2}\sqrt{D}$. To see this,
write $d=N^2+r$, $N=[\sqrt{d}]$,  $1\leq r \leq 2N$. If $r$ is even then  $A=r/2$ is  odd if $d\equiv 2,3$ mod 4 and $A=r/4$ is odd if $d\equiv 5$ mod 8.
Since  $d$ is a square mod $A$ we must have  $p\mid A\Rightarrow p\mid d$.  As $d$ is squarefree,  $A<\frac{1}{2}\sqrt{D}$  must be squarefree with  $A\mid d$,  giving 
$d=N^2+2A$ or $N^2+4A$ with $A\mid N$. 
Similarly for $r$  odd  
$$d=\left(\frac{r+1}{2}\right)^2+ \left(N-\frac{1}{2}(r-1)\right) \left(N+\frac{1}{2}(r-1)\right), $$
with $A=N-\frac{1}{2}(r-1)$ odd for $d=2$ or 3 mod 4 and $A=\frac{1}{2} \left(N-\frac{1}{2}(r-1)\right)$ odd
for $d\equiv 5$ mod 8. Since $d$ is a square mod A, $A<\frac{1}{2}\sqrt{D}$ is squarefree with $A\mid d$,  giving  $d=(N+1)^2-2A$ or $(N+1)^2-4A$ with $A\mid N+1$.

Conversely if $d$ is a quadratic residue mod $p$ for a suitably sized $p$  or if $d\equiv 1$ mod 8 then we can obtain 
a bound less than one for $L(d)/\sqrt{D}$:

\begin{lemma}\label{upperbound}
Suppose that  $d$ is a square mod $q$, where $q$ is odd or $4\mid q$ and $\lambda$ defined by
$$ \lambda \sqrt{d}=\begin{cases} q, & \text{ if $q$ is odd,}\\ \frac{1}{4}q, & \text{ if $q$ is even,} \end{cases}$$  has  $0<\lambda <1$.  Then
$$ \frac{L(d)}{\sqrt{D}} \leq \begin{cases} \frac{1}{2}(1+\lambda + \sqrt{(1-\lambda)^2 -4\lambda^2}), &  \text{ if }0< \lambda <  \frac{1}{4}(\sqrt{5}-1), \\
\frac{1}{4\lambda}, &  \text{ if } \frac{1}{4}(\sqrt{5}-1) <\lambda <\frac{1}{2}(\sqrt{3}-1) ,\\ 
\frac{1}{2}(1+\lambda), &   \text{ if }\frac{1}{2}(\sqrt{3}-1) <\lambda <1.
\end{cases} $$
\end{lemma}

Notice that if we assume GRH then estimate  \eqref{GRH}, with
$$\sum_{x-y \leq n\leq x} \Lambda(n)=y +O(x^{\frac{1}{2}}\log^2x)$$  from assuming RH, guarantees that $\left(\frac{d}{p}\right)=1$ for  some prime $p$ in \
$$\left(\frac{1}{2}(\sqrt{3}-1)d^{\frac{1}{2}}, \frac{1}{2}(\sqrt{3}-1)d^{\frac{1}{2}}+cd^{\frac{1}{4}}\log^2d\right)$$
for suitably large $c$, and Lemma \ref{upperbound} gives
$$ \frac{L(d)}{\sqrt{D}} \leq  \frac{1}{4}(\sqrt{3}+1)+ O\left(\frac{\log^2d}{d^{\frac{1}{4}}}\right) = 0.683012\ldots +o(1). $$

\begin{proof}
Suppose that  $r_0$  has $r_0^2\equiv d$ mod $q$.  If $q$ is odd we take $r$ to be the integer in
$(\sqrt{d}-2q,\sqrt{d})$ with the same parity as $d$ and $r\equiv r_0$ mod $q$, write $r=\sqrt{d}-\delta q$,
and set
$$ \alpha_1= \frac{\sqrt{d} +r}{2q}, \;\;\alpha_2= \frac{\sqrt{d} +r -2q}{2q}.$$
Notice that $c=2q$ and $a=r$ or $r-2q$ will have $c\mid (d-a^2)$ with $2c\mid d-a^2$
when $d=1$ mod 4.

For $4\mid q$  (which of course only occurs when $d=1$ mod 4) we write $q=2^lq_1$ with $q_1$ odd and $l\geq 2$ and take $r$ to be the integer in $(\sqrt{d}- 2^{l-1}q_1,\sqrt{d})$ with
$r\equiv r_0$ mod $2^{l-1}q_1$ and set
$$ \alpha_1= \frac{\sqrt{d} +r}{2^{l-1}q_1}, \;\;\alpha_2= \frac{\sqrt{d} +r -2^{l-1}q_1}{2^{l-1}q_1}. $$
Again $c=2^{l-1}q_1$ and $a=r$ or $r-2^{l-1}q_1$ will have $2c\mid (d-a^2)$.

Writing $r=\sqrt{d}-\delta q$ for  $q$ odd, and $r=\sqrt{d}-\delta 2^{l-2}q_1$
for $q$ even, we have $r=(1-\lambda \delta)\sqrt{d}$ with $0<\delta <2$ and
$$ \alpha_1= \frac{(2-\delta \lambda)}{2\lambda}, \;\; \overline{\alpha}_1=-\frac{\delta}{2}, \;\; 
\alpha_2= \frac{(2-2\lambda- \delta \lambda)}{2\lambda}, \;\; \overline{\alpha}_2=-\frac{\delta}{2}-1.$$
For $\alpha_1$ and $\alpha_2$ we also plainly have  $k=\epsilon c=2\epsilon \lambda \sqrt{d}=\lambda \sqrt{D}$.

Clearly  $\alpha_1>0$, $\alpha_2>-1$,  $-1<\overline{\alpha}_1 <0$ and  $-2< \overline{\alpha}_2<-1$.

If $\alpha_1<1$ then $M(\alpha_1)=\lambda \sqrt{D} <\frac{1}{2}(1+\lambda)\sqrt{D}$. Hence we can assume that $\alpha_1>1$ (this is automatic for $\lambda <\frac{1}{2}$).

 So
$$ M(\alpha_1)= \lambda \sqrt{D}\alpha_1 = \sqrt{D} \left( 1-\frac{\delta}{2}\lambda\right). $$
If  $\alpha_2<1$ then
$$ M(\alpha_2)= \lambda \sqrt{D} |\overline{\alpha}_2|= \sqrt{D} \lambda \left(1 +\frac{\delta}{2}\right), $$
and plainly 
$$ \min\{ M(\alpha_1) , M(\alpha_2)\} \leq \frac{1}{2} \left( M(\alpha_1)+M(\alpha_2)\right)= \frac{1}{2}(1+\lambda) \sqrt{D}. $$
So we can assume that $\alpha_2>1$ and
$$ M(\alpha_2) = \lambda \sqrt{D} |\overline{\alpha}_2|\alpha_2= \sqrt{D} \left(1+\frac{\delta}{2}\right)\left( 1-\lambda - \frac{\delta}{2}\lambda\right). $$
Observing that the quadratic is maximized for $\frac{\delta}{2}= \frac{1}{2\lambda}-1$ we plainly have
$$ M(\alpha_2) \leq \sqrt{D} \frac{1}{4\lambda}$$
with this less than  $\frac{1}{2}(1+\lambda)\sqrt{D}$
for $\frac{1}{2}(\sqrt{3}-1)<\lambda <1$. 
For $\lambda <\frac{1}{4}(\sqrt{5}-1)$ the value $\frac{\delta}{2}= \frac{1}{2\lambda}\left(1-\lambda -\sqrt{ (1-\lambda)^2-4\lambda^2}\right)$ equating $M(\alpha_1)$
and $M(\alpha_2)$ is less than $\frac{1}{2\lambda}-1$ and the minimum of the two is at most
the value at that point:
$$ \min\{M(\alpha_1),M(\alpha_2)\} \leq \sqrt{D} \left( \frac{1}{2}(1+\lambda)+\frac{1}{2}\sqrt{ (1-\lambda)^2-4\lambda^2}\right). $$ 
\end{proof}

In particular from the lemma we immediately obtain a bound away from 1 for the $d\equiv 1$ mod 8.

\begin{corollary}
If $d\equiv 1$ (mod 8) then 
$$\frac{L(d)}{\sqrt{D}} \leq \frac{1}{4}(\sqrt{5}+1)=0.809016\ldots  . $$ 
If $d\equiv 1$ (mod 3) then
$$\frac{L(d)}{\sqrt{D}} \leq \frac{1}{7}(2+3\sqrt{2})=0.891805 \ldots  . $$
\end{corollary}

Computations indicate room for improvement in these bounds.

\begin{conjecture}
\begin{align*} 
\sup_{d\equiv 1 \text{ mod } 8}  \frac{L(d)}{\sqrt{D}} & =\frac{L(41)}{\sqrt{41}}=\frac{M\left(\frac{\sqrt{41}+27}{4}\right)}{\sqrt{41}}=\frac{\sqrt{41}+3}{2\sqrt{41}}=0.734261\dots,\\
\sup_{d\equiv 1 \text{ mod } 3}  \frac{L(d)}{\sqrt{D}}   & =\frac{L(13)}{\sqrt{13}}=\frac{M\left(\frac{\sqrt{13}+1}{2}\right)}{\sqrt{13}}=\frac{4}{\sqrt{13}}=0.832050\ldots  .
\end{align*}
\end{conjecture}

\begin{proof}
If $d\equiv 1$ (mod 8) then we can solve $r^2\equiv d$ mod $2^{l}$ for any $l$. 
Hence if  we  pick $l$  such that 
 $\frac{1}{4}(\sqrt{5}-1)\sqrt{d} \leq 2^{l-2} \leq \frac{1}{2}(\sqrt{5}-1)\sqrt{d}$ and
we can apply the lemma with  $\frac{1}{4}(\sqrt{5}-1)\leq \lambda \leq  \frac{1}{2}(\sqrt{5}-1)$.
Likewise, for an odd prime $p$, if $p\nmid d$ and $\left(\frac{d}{p}\right)=1$ then we can solve
$r^2\equiv d$ mod $p^{l}$ for any $l$.  Choosing $l$ so that 
$$ \frac{1}{1+\sqrt{(p-1)^2+4}}\sqrt{d} \leq p^l \leq  \frac{p}{1+\sqrt{(p-1)^2+4}}\sqrt{d},  $$
and applying the lemma with $q=p^l$ gives
$$ \frac{L(d)}{\sqrt{D}}\leq \frac{1}{2} \left( 1+ \frac{p}{1+\sqrt{(p-1)^2+4}}\right). $$
Taking  $p=3$ gives the result claimed for   $d\equiv 1$ mod 3. 

Likewise, for  $d\equiv 1,4$ mod 5 we get the upper bound $0.956859\ldots$ (from $d=29$ we know  $0.928476\ldots$ would be best possible).

For $ d\equiv 1,2,4$ mod 7 we get $0.977844\ldots $   and for $d\equiv 1,3,4,5,9$ mod 11 the  bound $0.991157\ldots $   (from $d=53$ these can not be reduced below $0.961523 \ldots$).

For $d\equiv 1,3,4,9,10,12$ mod 13 we get the  bound $0.993713\ldots $  (the optimal bound is likely $ 0.988371\ldots$ from $d= 173$).

For $d\equiv 1,2,4,8,9,13,15,16$ mod 17 our bound gives  $0.996364\ldots $  (optimal is probably  $0.993150\ldots $  at  $d=293$).
\end{proof}


\section{Appendix of Large Values}

We give the largest values found  in Figure \ref{fig:onemod4s}, Figure \ref{fig:twomod4s} \& 
Figure \ref{fig:threemod4s}  down to the first value 
not satisfying the quadratic non-residue  conditions of  Lemma \ref{nonresidue}.

\vspace{2ex}
\noindent
{\bf Largest values for $d\equiv 1 $ mod 4}.
\begin{align*}
\frac{L(293)}{\sqrt{293}} & =\frac{ M\left( \frac{\sqrt{293}+15}{2}\right)}{\sqrt{293}} = \frac{17}{\sqrt{293}}=0.993150\ldots, \;\; \; \left(\frac{293}{p}\right)  =-1, p=3,5,7,11,13,\\
\frac{L(173)}{\sqrt{173}} & =\frac{ M\left( \frac{\sqrt{173}+11}{2}\right)}{\sqrt{173}} = \frac{13}{\sqrt{173}}=0.988371\ldots, \;\;\; \left(\frac{173}{p}\right)=-1, p=3,5,7,11,  \\
\frac{L(53)}{\sqrt{53}} & =\frac{ M\left( \frac{\sqrt{53}+5}{2}\right)}{\sqrt{53}} = \frac{7}{\sqrt{53}}=0.961523 \ldots,\;\;\;   \left(\frac{53}{p}\right) =-1, \;p=3,5, \\
\frac{L(437)}{\sqrt{437}} & =\frac{ M\left( \frac{\sqrt{437}+19}{2}\right)}{\sqrt{437}}  = \frac{\sqrt{437}+19}{2\sqrt{437}}=0.954446\ldots , \;\;\; \left(\frac{437}{p}\right) =-1, \;p=3,5,7,11,13,17,29 ,   \\ 
\frac{L(29)}{\sqrt{29}} & =\frac{ M\left( \frac{\sqrt{29}+3}{2}\right)}{\sqrt{29}} = \frac{5}{\sqrt{29}}=0.928476 \ldots,\;\;\; \left(\frac{29}{p}\right) =-1, \;p=3, \\  
\frac{L(1077)}{\sqrt{1077}} & =\frac{ M\left( \frac{\sqrt{1077}+27}{6}\right)}{\sqrt{1077}} = \frac{\sqrt{1077}+27}{2\sqrt{1077}}=0.911363 \ldots,\;\;\; \left(\frac{1077}{p}\right)  =-1, \;p=5,7,11,13,17,19,23, \\
\frac{L(77)}{\sqrt{77}} & =\frac{ M\left( \frac{\sqrt{77}+7}{2}\right)}{\sqrt{77}} = \frac{\sqrt{77}+7}{\sqrt{77}}=0.898862 \ldots,\;\;\; \left(\frac{77}{p}\right)  =-1, \;p=3,5, \\
\frac{L(453)}{\sqrt{453}} & =\frac{ M\left( \frac{\sqrt{453}+15}{6}\right)}{\sqrt{453}} = \frac{19}{\sqrt{453}}=0.892697 \ldots,\;\;\;\left(\frac{453}{p}\right) =-1, \;p=5,7,11,13,17, \\
\frac{L(717)}{\sqrt{717}} & =\frac{ M\left( \frac{\sqrt{717}+21}{6}\right)}{\sqrt{717}} = \frac{\sqrt{717}+21}{\sqrt{717}}=0.892129\ldots,\;\;\;\left(\frac{717}{p}\right)  =-1, \;p=5,7,11,13,17, 19,\\
\frac{L(3053)}{\sqrt{3053}} & =\frac{ M\left( \frac{\sqrt{3053}+41}{14}\right)}{\sqrt{3053}} = \frac{49}{\sqrt{3053}}=0.886814\ldots,\;\;\; \left(\frac{3053}{7}\right)  =1.
\end{align*}
Note other $\alpha$ may achieve the minimum, for example $M\left( \frac{\sqrt{437}+19}{2}\right)  = M\left( \frac{\sqrt{437}+19}{38}\right)$.

\newpage

\vspace{2ex}
\noindent
{\bf Largest values for $d\equiv 2$ mod 4}.
\begin{align*}
\frac{L(398)}{2\sqrt{398}} & =\frac{ M\left( \frac{\sqrt{398}+18}{2}\right)}{2\sqrt{398}}=\frac{\sqrt{398}+ 18}{2\sqrt{398}}=0.951129\ldots,\;\;\; \left(\frac{398}{p}\right) =-1, \;p=3,5,7,11,13,17,19,23,29,31, \\  
\frac{L(38)}{2\sqrt{38}} & =\frac{ M\left( \frac{\sqrt{38}+4}{2}\right)}{2\sqrt{38}}=\frac{11}{2\sqrt{38}}=0.892217\ldots,\;\;\; \left(\frac{38}{p}\right) =-1, \;p=3,5, \\  
\frac{L(62)}{2\sqrt{62}} & =\frac{ M\left( \frac{\sqrt{62}+6}{2}\right)}{2\sqrt{62}}=\frac{\sqrt{62}+ 6}{2\sqrt{62}}=0.881000\ldots,\;\;\; \left(\frac{62}{p}\right) =-1, \;p=3,5,7,11, \\  
\frac{L(318)}{2\sqrt{318}} & =\frac{ M\left( \frac{\sqrt{318}+12}{6}\right)}{2\sqrt{318}}=\frac{\sqrt{318}+ 12}{2\sqrt{318}}=0.836463\ldots,\;\;\; \left(\frac{318}{p}\right) =-1, \;p=5,7,11,13,17,19,23, \\  
\frac{L(14)}{2\sqrt{14}} & =\frac{ M\left( \frac{\sqrt{14}+2}{2}\right)}{2\sqrt{14}}=\frac{\sqrt{14}+ 2}{2\sqrt{14}}=0.767261\ldots,\;\;\; \left(\frac{14}{p}\right) =-1, \;p=3,\\  
\frac{L(138)}{2\sqrt{138}} & =\frac{ M\left( \frac{\sqrt{138}+6}{6}\right)}{2\sqrt{138}}=\frac{\sqrt{138}+ 6}{2\sqrt{138}}=0.755376\ldots,\;\;\; \left(\frac{138}{p}\right) =-1, \;p=5,7,11,13, \\  
\frac{L(22)}{2\sqrt{22}} & =\frac{ M\left( \frac{\sqrt{22}+2}{3}\right)}{2\sqrt{22}}=\frac{\sqrt{22}+ 2}{2\sqrt{22}}=0.713200\ldots,\;\;\; \left(\frac{22}{3}\right) =1. 
\end{align*}

\newpage 

\vspace{2ex}
\noindent
{\bf Largest values for $d\equiv 3$ mod 4}.
\begin{align*}
\frac{L(227)}{2\sqrt{227}} & =\frac{ M\left( \frac{\sqrt{227}+13}{2}\right)}{2\sqrt{227}} = \frac{29}{2\sqrt{227}}=0.962398 \ldots ,\;\;\;\left(\frac{227}{p}\right) =-1, \;p=3,5,7,11,13,17,19,23,  \\
\frac{L(83)}{2\sqrt{83}} & =\frac{ M\left( \frac{\sqrt{83}+7}{2}\right)}{2\sqrt{83}} = \frac{17}{2\sqrt{83}}=0.932966 \ldots,\;\;\;  \left(\frac{83}{p}\right)  =-1, \;p=3,5,7,11,13,\\
\frac{L(167)}{2\sqrt{167}} & =\frac{ M\left( \frac{\sqrt{167}+11}{2}\right)}{2\sqrt{167}} = \frac{\sqrt{167}+11}{2\sqrt{167}}=0.925602 \ldots,\;\;\; \left(\frac{167}{p}\right)  =-1, \;p=3,5,7,11,13,17,19,  \\
\frac{L(447)}{2\sqrt{447}} & =\frac{ M\left( \frac{\sqrt{447}+15}{6}\right)}{2\sqrt{447}} = \frac{37}{2\sqrt{447}}=0.875019 \ldots ,\;\;\;\left(\frac{447}{p}\right) =-1, \;p=5,7,11,13,17,  \\
\frac{L(47)}{2\sqrt{47}} & =\frac{ M\left( \frac{\sqrt{47}+5}{2}\right)}{2\sqrt{47}} = \frac{\sqrt{47}+5}{2\sqrt{47}}=0.864662 \ldots ,\;\;\;\left(\frac{47}{p}\right) =-1, \;p=3,5,7,  \\
\frac{L(635)}{2\sqrt{635}} & =\frac{ M\left( \frac{\sqrt{635}+15}{10}\right)}{2\sqrt{635}} = \frac{41}{2\sqrt{635}}=0.813517 \ldots ,\;\;\;\left(\frac{635}{p}\right) =-1, \;p=3,7,11,13,17,19,23,29,31,37,  \\
\frac{L(23)}{2\sqrt{23}} & =\frac{ M\left( \frac{\sqrt{23}+3}{2}\right)}{2\sqrt{23}} = \frac{\sqrt{23}+3}{2\sqrt{23}}=0.812771\ldots ,\;\;\;\left(\frac{23}{p}\right) =-1, \;p=3,5,  \\
\frac{L(3)}{2\sqrt{3}} & =\frac{ M\left( \frac{\sqrt{3}+1}{2}\right)}{2\sqrt{3}} = \frac{\sqrt{3}+1}{2\sqrt{3}}=0.788675 \ldots ,\;\;\;\left(\frac{3}{p}\right) =-1, \;p=3,7,11,13,17,19,23,29,31,37,  \\
\frac{L(827)}{2\sqrt{827}} & =\frac{ M\left( \frac{\sqrt{827}+15}{14}\right)}{2\sqrt{827}} = \frac{\sqrt{827}+15}{2\sqrt{827}}=0.760800 \ldots ,\;\;\;\left(\frac{827}{7}\right) =1.  \\
\end{align*}

\end{document}